\documentclass{amsart}
\usepackage{amssymb,latexsym}
\usepackage{mathrsfs}
\usepackage{epic}
\begin{document}

\title[Betti diagrams of multigraded artinian modules]{The linear
space of Betti diagrams of 
multigraded artinian modules}

\author{Gunnar Fl{\o}ystad}
\address{Matematisk Institutt\\
         University of Bergen \\
         Johs. Brunsgt. 12\\
         5008 Bergen \\
         Norway}
\email{gunnar@mi.uib.no}

\keywords{keywords}
\subjclass[2000]{Primary: 13D02; Secondary: 13C14.}
\date{\today}


\theoremstyle{plain}
\newtheorem{theorem}{Theorem}[section]
\newtheorem{corollary}[theorem]{Corollary}
\newtheorem*{main}{Main Theorem}
\newtheorem{lemma}[theorem]{Lemma}
\newtheorem{proposition}[theorem]{Proposition}
\newtheorem{conjecture}[theorem]{Conjecture}
\newtheorem*{theoremA}{Theorem}
\newtheorem*{theoremB}{Theorem}

\theoremstyle{definition}
\newtheorem{definition}[theorem]{Definition}
\newtheorem{question}[theorem]{Question}

\theoremstyle{remark}
\newtheorem{notation}[theorem]{Notation}
\newtheorem{remark}[theorem]{Remark}
\newtheorem{example}[theorem]{Example}
\newtheorem{claim}{Claim}


\newcommand{\psp}[1]{{{\bf P}^{#1}}}
\newcommand{\psr}[1]{{\bf P}(#1)}
\newcommand{\op}{{\mathcal O}}
\newcommand{\opw}{\op_{\psr{W}}}
\newcommand{\go}{\op}

\newcommand{\ini}[1]{\text{in}(#1)}
\newcommand{\gin}[1]{\text{gin}(#1)}
\newcommand{\kr}{{\Bbbk}}
\newcommand{\kk}{{\Bbbk}}
\newcommand{\pd}{\partial}
\newcommand{\vardel}{\partial}
\renewcommand{\tt}{{\bf t}}


\newcommand{\coh}{{{\text{{\rm coh}}}}}


\newcommand{\modv}[1]{{#1}\text{-{mod}}}
\newcommand{\modstab}[1]{{#1}-\underline{\text{mod}}}

\newcommand{\sut}{{}^{\tau}}
\newcommand{\sumit}{{}^{-\tau}}
\newcommand{\til}{\thicksim}

\newcommand{\totp}{\text{Tot}^{\prod}}
\newcommand{\dsum}{\bigoplus}
\newcommand{\dprod}{\prod}
\newcommand{\lsum}{\oplus}
\newcommand{\lprod}{\Pi}

\newcommand{\La}{{\Lambda}}
\newcommand{\lam}{{\lambda}}
\newcommand{\GL}{{GL}}

\newcommand{\sirstj}{\circledast}

\newcommand{\she}{\EuScript{S}\text{h}}
\newcommand{\cm}{\EuScript{CM}}
\newcommand{\cmd}{\EuScript{CM}^\dagger}
\newcommand{\cmri}{\EuScript{CM}^\circ}
\newcommand{\cler}{\EuScript{CL}}
\newcommand{\clerd}{\EuScript{CL}^\dagger}
\newcommand{\clerri}{\EuScript{CL}^\circ}
\newcommand{\gor}{\EuScript{G}}
\newcommand{\gF}{\mathcal{F}}
\newcommand{\gG}{\mathcal{G}}
\newcommand{\gM}{\mathcal{M}}
\newcommand{\gE}{\mathcal{E}}
\newcommand{\gD}{\mathcal{D}}
\newcommand{\gI}{\mathcal{I}}
\newcommand{\gP}{\mathcal{P}}
\newcommand{\gK}{\mathcal{K}}
\newcommand{\gL}{\mathcal{L}}
\newcommand{\gS}{\mathcal{S}}
\newcommand{\gC}{\mathcal{C}}
\newcommand{\gO}{\mathcal{O}}
\newcommand{\gJ}{\mathcal{J}}
\newcommand{\gU}{\mathcal{U}}
\newcommand{\mm}{\mathfrak{m}}

\newcommand{\dlim} {\varinjlim}
\newcommand{\ilim} {\varprojlim}

\newcommand{\CM}{\text{CM}}
\newcommand{\Mon}{\text{Mon}}


\newcommand{\Kom}{\text{Kom}}


\newcommand{\EH}{{\mathbf H}}
\newcommand{\res}{\text{res}}
\newcommand{\Hom}{\text{Hom}}
\newcommand{\inhom}{{\underline{\text{Hom}}}}
\newcommand{\Ext}{\text{Ext}}
\newcommand{\Tor}{\text{Tor}}
\newcommand{\ghom}{\mathcal{H}om}
\newcommand{\gext}{\mathcal{E}xt}
\newcommand{\id}{\text{{id}}}
\newcommand{\im}{\text{im}\,}
\newcommand{\codim} {\text{codim}\,}
\newcommand{\resol}{\text{resol}\,}
\newcommand{\rank}{\text{rank}\,}
\newcommand{\lpd}{\text{lpd}\,}
\newcommand{\coker}{\text{coker}\,}
\newcommand{\supp}{\text{supp}\,}
\newcommand{\Ad}{A_\cdot}
\newcommand{\Bd}{B_\cdot}
\newcommand{\Fd}{F_\cdot}
\newcommand{\Gd}{G_\cdot}


\newcommand{\sus}{\subseteq}
\newcommand{\sups}{\supseteq}
\newcommand{\pil}{\rightarrow}
\newcommand{\vpil}{\leftarrow}
\newcommand{\rpil}{\leftarrow}
\newcommand{\lpil}{\longrightarrow}
\newcommand{\inpil}{\hookrightarrow}
\newcommand{\pils}{\twoheadrightarrow}
\newcommand{\projpil}{\dashrightarrow}
\newcommand{\dotpil}{\dashrightarrow}
\newcommand{\adj}[2]{\overset{#1}{\underset{#2}{\rightleftarrows}}}
\newcommand{\mto}[1]{\stackrel{#1}\longrightarrow}
\newcommand{\vmto}[1]{\overset{\tiny{#1}}{\longleftarrow}}
\newcommand{\mtoelm}[1]{\stackrel{#1}\mapsto}

\newcommand{\eqv}{\Leftrightarrow}
\newcommand{\impl}{\Rightarrow}

\newcommand{\iso}{\cong}
\newcommand{\te}{\otimes}
\newcommand{\into}[1]{\hookrightarrow{#1}}
\newcommand{\ekv}{\Leftrightarrow}
\newcommand{\equi}{\simeq}
\newcommand{\isopil}{\overset{\cong}{\lpil}}
\newcommand{\equipil}{\overset{\equi}{\lpil}}
\newcommand{\ispil}{\isopil}
\newcommand{\vvi}{\langle}
\newcommand{\hvi}{\rangle}
\newcommand{\susneq}{\subsetneq}
\newcommand{\sgn}{\text{sign}}
\newcommand{\prikk}{\bullet}


\newcommand{\xd}{\check{x}}
\newcommand{\ortog}{\bot}
\newcommand{\tL}{\tilde{L}}
\newcommand{\tM}{\tilde{M}}
\newcommand{\tH}{\tilde{H}}
\newcommand{\tvH}{\widetilde{H}}
\newcommand{\tvh}{\widetilde{h}}
\newcommand{\tV}{\tilde{V}}
\newcommand{\tS}{\tilde{S}}
\newcommand{\tT}{\tilde{T}}
\newcommand{\tR}{\tilde{R}}
\newcommand{\tf}{\tilde{f}}
\newcommand{\ts}{\tilde{s}}
\newcommand{\tp}{\tilde{p}}
\newcommand{\tr}{\tilde{r}}
\newcommand{\tfst}{\tilde{f}_*}
\newcommand{\empt}{\emptyset}
\newcommand{\bfa}{{\bf a}}
\newcommand{\bfb}{{\bf b}}
\newcommand{\bfd}{{\bf d}}
\newcommand{\bfe}{{\bf e}}
\newcommand{\bfp}{{\bf p}}
\newcommand{\bfc}{{\bf c}}
\newcommand{\bfl}{{\bf t}}
\newcommand{\la}{\lambda}
\newcommand{\bfen}{{\mathbf 1}}
\newcommand{\ep}{\epsilon}
\newcommand{\en}{r}
\newcommand{\tu}{s}
\newcommand{\carc}{\mbox{char.}}

\newcommand{\ome}{\omega_E}

\newcommand{\bevis}{{\bf Proof. }}
\newcommand{\demofin}{\qed \vskip 3.5mm}
\newcommand{\nyp}[1]{\noindent {\bf (#1)}}
\newcommand{\demo}{{\it Proof. }}
\newcommand{\demodone}{\demofin}
\newcommand{\parg}{{\vskip 2mm \addtocounter{theorem}{1}  
                   \noindent {\bf \thetheorem .} \hskip 1.5mm }}

\newcommand{\red}{{\text{red}}}
\newcommand{\lcm}{{\text{lcm}}}


\newcommand{\dl}{\Delta}
\newcommand{\cdel}{{C\Delta}}
\newcommand{\cdelp}{{C\Delta^{\prime}}}
\newcommand{\dlst}{\Delta^*}
\newcommand{\Sdl}{{\mathcal S}_{\dl}}
\newcommand{\lk}{\text{lk}}
\newcommand{\lkd}{\lk_\Delta}
\newcommand{\lkp}[2]{\lk_{#1} {#2}}
\newcommand{\del}{\Delta}
\newcommand{\delr}{\Delta_{-R}}
\newcommand{\dd}{{\dim \del}}

\renewcommand{\aa}{{\bf a}}
\newcommand{\bb}{{\bf b}}
\newcommand{\cc}{{\bf c}}
\newcommand{\xx}{{\bf x}}
\newcommand{\yy}{{\bf y}}
\newcommand{\zz}{{\bf z}}
\newcommand{\mv}{{\xx^{\aa_v}}}
\newcommand{\mF}{{\xx^{\aa_F}}}

\newcommand{\pnm}{{\bf P}^{n-1}}
\newcommand{\opnm}{{\go_{\pnm}}}
\newcommand{\ompnm}{\omega_{\pnm}}

\newcommand{\pn}{{\bf P}^n}
\newcommand{\hele}{{\mathbb Z}}
\newcommand{\nat}{{\mathbb N}}
\newcommand{\rasj}{{\mathbb Q}}

\newcommand{\dt}{{\displaystyle \cdot}}
\newcommand{\st}{\hskip 0.5mm {}^{\rule{0.4pt}{1.5mm}}}              
\newcommand{\disk}{\scriptscriptstyle{\bullet}}

\newcommand{\cF}{F_\dt}
\newcommand{\pol}{f}

\newcommand{\disc}{\circle*{5}}

\def\CC{{\mathbb C}}
\def\GG{{\mathbb G}}
\def\ZZ{{\mathbb Z}}
\def\NN{{\mathbb N}}
\def\RR{{\mathbb R}}
\def\OO{{\mathbb O}}
\def\QQ{{\mathbb Q}}
\def\VV{{\mathbb V}}
\def\PP{{\mathbb P}}
\def\EE{{\mathbb E}}
\def\FF{{\mathbb F}}
\def\AA{{\mathbb A}}

\begin{abstract}
We study the linear space generated by the
multigraded Betti diagrams of $\hele^n$-graded artinian
modules of codimension $n$ whose resolutions become pure of a given type
when taking total degrees. We show that 
the multigraded Betti diagram of the equivariant resolution constructed
in \cite{EFW} by D.Eisenbud, J.Weyman, and the author, and all its
twists, form a basis for this linear space. We also show that it is
essentially unique with this property.
\end{abstract}
\maketitle

\section*{Introduction}

Recent years has seen a breakthrough in the studies of syzygies of graded
modules over the polynomial ring $S = \kk[x_1, \ldots, x_n]$. 
In \cite{BS}, M.Boij and J.S\"oderberg formulated conjectures describing the positive
cone of Betti diagrams of artinian modules over the polynomial ring. The conjectures
were subsequently proven by the work of D.Eisenbud, J.Weyman, and the author in 
\cite{EFW}, and by Eisenbud and F.-O. Schreyer in \cite{ES}. Fundamental in showing
the conjectures is to show the existence of pure resolution of artinian modules.
A resolution is pure if it has the form
\[ S(-d_0)^{\beta_0} \vpil S(-d_1)^{\beta_1} 
\vpil \cdots \vpil S(-d_n)^{\beta_n} \]
for a sequence $\bfd : d_0 < d_1 < \cdots < d_n$. In \cite{EFW} the existence 
of such
resolutions of graded artinian modules is shown for every sequence $\bfd$ when
$\mbox{char}\, \kk = 0$.  Moreover the construction there is a quite explicit 
$\GL(n)$-equivariant  resolution. In particular it is equivariant for the diagonal matrices
and hence ${\hele}^n$-graded.

 The beauty and naturality of this resolution is apparent from the construction.
It has recently been generalised by S.Sam and Weyman in \cite{SW} to wider classes
of equivariant resolutions. In this paper we consider the class of resolutions of 
${\hele}^n$-graded artinian modules which become pure when taking total degrees.
We establish, in a precise sense, that the equivariant resolution constructed in \cite{EFW}, or rather
its multigraded Betti diagram, is the fundamental resolution in this class. 

   Before stating the results more precisely, let us consider a simple example
for illustration. If $F_{\prikk}$  is a resolution of a  ${\hele}^n$-graded module,
we have a twisted complex $F_{\prikk}(\bfa)$ for $\bfa \in \hele^n$. If $\beta$
is the Betti diagram of $F_\prikk$, then $F_{\prikk}(\bfa)$ will have a Betti diagram
which we denote by $\beta(\bfa)$. 

\begin{example} The following example was worked out together with J.Weyman.
Let $S =\kk[x_1,x_2]$ and suppose
$d_1 - d_0 = 2$ and $d_2 - d_1 = 3$. 
The equivariant resolution has the following form where we have written
the bidegrees of the generators below the terms.
\begin{equation} \label{IntroLigEkvi}
\underset{\scriptsize{\begin{matrix} (2,0) \\ (1,1) \\ (0,2) \end{matrix}}} {S^3} \vpil 
\underset{\scriptsize{\begin{matrix} (4,0) \\ (3,1) \\ (2,2) \\ (1,3) \\ (0,4) \end{matrix}}}
{S^5} \vpil
\underset{\scriptsize{\begin{matrix} (4,3) \\ (3,4) \end{matrix}}} {S^2}.
\end{equation}
Let $\beta_1$ be its bigraded Betti table.
In \cite[Remark 3.2]{BS} Boij and S\"oderberg also gave 
a construction of pure resolutions in the case of two variables.
These were resolutions of a quotient of a pair of monomial ideals. 
For the type above the resolution had the following
bidegrees.
\begin{equation} \label{IntroLigBS}
\underset{\scriptsize{\begin{matrix} (4,0) \\ (2,2) \\ (0,4) \end{matrix}}} 
{S^3} \vpil 
\underset{\scriptsize{\begin{matrix} (6,0) \\ (4,2) \\ (3,3) \\ (2,4) \\ (0,6) 
\end{matrix}}} {S^5} \vpil
\underset{\scriptsize{\begin{matrix} (6,3) \\ (3,6) \end{matrix}}} {S^2}.
\end{equation}
Let $\beta_2$ be its Betti diagram. Then $\beta_2$ is a linear
combination of twists of $\beta_1$ but not vice versa :
$\beta_2 = \beta_1(2,0) - \beta_1(1,1) + \beta_1(0,2)$.
This indicates that in some way the complex (\ref{IntroLigEkvi}), or
at least its Betti diagram, is more fundamental than that of the complex 
(\ref{IntroLigBS}). 
\end{example}

Now let $e_i = d_i - d_{i-1}$ and let $\bfe$ be the sequence of these differences.
Consider ${\hele}^n$-graded resolutions and their Betti diagrams, 
of artinian ${\hele}^n$-graded modules
over the polynomial ring $S$ of dimension $n$.
We shall be interested in those resolutions which become pure when taking total degrees
and for which the difference sequence of these total degrees is $\bfe$.
Let $L(\bfe)$ be the $\rasj$-vector space generated by such Betti diagrams. Our result is
a complete description of this vector space when $\kr$ has characteristic $0$. 
For the given $\bfe$, consider the equivariant resolution constructed in \cite{EFW},
which has $\bfe$ as difference vector of the total degrees, and let
$\beta$ (which of course depends on $\bfe$) be its ${\hele}$-graded Betti
diagram.
Also let $r$ be the greatest common divisor
of $e_1, \ldots, e_n$. In the case $r=1$ our main result Theorem 
\ref{LinbettiTheMainTo} says the following.
\begin{theoremA}
The set of twisted diagrams
$\beta(\bfa)$ where $\bfa \in \hele^n$, constitute a basis for 
the lattice of integral points in $L(\bfe)$. Moreover, up to a twist $\bfa$, 
$\beta$ is the unique Betti diagram with this property.

In particular the $\beta(\bfa)$ where $\bfa \in \hele^n$ form a basis for the
vector space $L(\bfe)$.
\end{theoremA}

This theorem shows the canonical stature of the multigraded Betti diagram
of the equivariant resolution.

For arbitrary $r$ the theorem holds true but with $\beta$ being the Betti diagram of 
a somewhat modified resolution.
We consider the equivariant resolution
with difference sequence $\bfe^\prime$ where $\bfe = r \cdot \bfe^\prime$.
By replacing $x_i$ by $x_i^r$ in this resolution corresponding to $\bfe^\prime$, 
we obtain a resolution with
difference sequence $\bfe$, and now let $\beta$ be the Betti diagram of this
complex. With this modification the above theorem holds for any $r$ (when 
$\kr$ has characteristic $0$).


The Betti diagrams of $\ZZ^n$-graded artinian modules fulfil a multigraded
version of the Herzog-K\"uhl equations. We then introduce the $\rasj$-vector space
$L^\prime(\bfe)$ generated by $\ZZ^n$-graded diagrams (that need not
arise from resolutions) whose total degrees are pure with difference vector $\bfe$,
and which fulfil the multigraded Herzog-K\"uhl equations. Note that the field
$\kr$ is not involved in the definition of this linear space.

There is a natural injection $L(\bfe) \pil L^\prime(\bfe)$. With $L(\bfe)$ replaced
by $L^\prime(\bfe)$, the theorem above also holds and is the essential part of our 
first main Theorem
\ref{LinbettiTheMainEn} (this does not involve $\kr$).
This statement will imply that $L^\prime(\bfe)$ may be identified with the Laurent
polynomial ring ${\QQ}\{t_1, \ldots, t_n\}$, and $L(\bfe)$ may be identified as an
ideal ${\mathscr I}(\kr;\bfe)$ in this ring. When $\kr$ has characteristic $0$, what we say
above shows that this ideal is the whole ring.
But when it has characteristic $p$ it is an intriguing question,
which we do not know much about, to describe this ideal. 

\medskip
In the case where the resolutions are simply $\hele$-graded instead
of $\hele^n$-graded, one may also consider the vector space $L(\bfe)$. This decomposes
as a sum of one-dimensional spaces $L(\bfd)$, one for each $\bfd$ with
difference sequence equal to $\bfe$.
Taking the lattice of integer points in $L(\bfd)$, it is conjectured in \cite{EFW} 
that all sufficiently large diagrams may be realised as Betti diagrams of resolutions.
There are however positive diagrams which are not realised
by any resolution, see also \cite{Er} and \cite{EES} for more in this direction.  
Our result implies, see Corollary \ref{linbettiCorSimgr}, that in 
the $\hele^n$-graded case, if you take an integer lattice point of 
$L(\bfd)$, a $\ZZ^n$-graded diagram, and form the $\ZZ$-graded diagram 
from it by taking total degrees, then such a diagram must be an integer
multiple of the $\ZZ$-graded Betti diagram associated to the equivariant
resolution.

Of course, even more interesting than the linear space $L(\bfe)$ is the positive 
rational cone
$P(\bfe)$ generated by the Betti diagrams of resolutions of ${\hele}^n$-graded 
artinian
modules which are pure with respect to total degrees and with difference sequence
$\bfe$ 
of the total degrees. It is considerably more difficult to describe this cone. In
the paper \cite{BF} by Boij and the author, we describe this cone completely
in the case of two variables, and give some examples in the case of three variables.

\medskip
The organisation of the paper is as follows. In Section 1 we give
basic facts and notations. We note that multigraded Betti diagrams fulfil
strong numerical criteria, the multigraded versions of the Herzog-K\"uhl equations.
We recall the form of the equivariant complex, and state our main results concerning the
basis of the linear space $L^\prime(\bfe)$, and that $L(\bfe)$ identifies with
this space when $\kr$ has characteristic $0$. 
The terms of the equivariant resolution are of the form $S \te_k S_\la$
for a Schur module $S_\la$. In Section 2 we study the associated
Schur polynomials $s_\la$ of the terms in the resolution. In order to 
establish our main result, we find the greatest common divisor of these
polynomials. In Section 3 we describe
the structure of the linear space $L^\prime(\bfe)$. From this description and the results in 
Section 2 concerning Schur polynomials, we give the (immediate) proofs of the main theorems 
describing $L(\bfe)$ and $L^\prime(\bfe)$.

\section{The linear space of multigraded Betti diagrams}

In this section we give the basic facts and notations concerning
multigraded Betti diagrams. 
We describe the multigraded Herzog-K\"uhl equations.
We recall the 
construction of pure resolutions in \cite{EFW}. In the end
we give the statement of our main result.

\subsection{Betti diagrams and the Herzog-K\"uhl equations}
\label{SetSubsHK}

Let $S = \kr[x_1, \ldots, x_n]$ be the polynomial ring over a field $\kr$.
We shall study $\hele^ n$-graded free resolutions of artinian 
$\hele^ n$-graded $S$-modules
\[ F_0 \vpil F_1 \vpil \cdots \vpil F_n.  \]
 For a multidegree $\bfa = (a_1, a_2, \ldots, a_n)$ in $\hele^ n$
let $|\bfa| = \sum a_i$ be its total degree. We shall be interested
in the case that these resolutions become pure if we make
them singly graded by taking total degrees. So there is a
sequence $d_0 < d_1 < \cdots < d_n$ such that 
\[ F_i = \oplus_{|\bfa| = d_i} S(- \bfa)^ {\beta_{i, \bfa}}. \] 
The {\it multigraded Betti diagram} of such a resolution is the element
\[ \{\beta_{i,\bfa}\}_{\scriptsize \underset{}
{\begin{matrix} i=0, \ldots, n \\ \bfa \in \hele^n \end{matrix}  }}
\in \oplus_{\hele^n} \nat^{n+1}. \] 

A way of representing
a multigraded Betti diagram which will turn out very convenient for us, 
is to represent $\beta = \{ \beta_{i, \bfa} \}$
where $i = 0,\ldots,n$ and $\bfa \in \hele^{n}$ by Laurent polynomials
\[ B_i(t) = \sum_{\bfa \in \hele^n} \beta_{i,\bfa} \cdot t^\bfa. \]
We call this the {\it Betti polynomial} of the diagram $\beta$ or the 
resolution $F_\prikk$. 
We thus get an $(n+1)$-tuple of Laurent polynomials
\[ B = (B_0, B_1, \ldots, B_n). \]

\medskip
Given a set of total degrees $\bfd : d_0 < d_1 < \cdots < d_n$.
Let 
$L(\bfd)$ in $\oplus_{\hele^n} \rasj^ {n+1}$ be the linear subspace 
generated by
multigraded Betti diagrams of artinian $\hele^n$-graded modules
whose resolutions become pure of degrees $d_0, \ldots, d_n$ after
taking total degrees.  

Furthermore  let $e_i = d_i - d_{i-1}$.
This gives the difference vector $\Delta \bfd = \bfe = (e_1, \ldots, e_n)$.
Most of the time it will be convenient to fix the difference vector
instead of the vector of total degrees. 
We therefore let $L(\bfe) = \oplus_{\Delta \bfd  = \bfe} L(\bfd)$ be the linear
subspace of  $\oplus_{\hele^n} \rasj^ {n+1}$ 
generated by all multigraded Betti diagrams which become pure when
considering total degrees, and where the difference vector of these total
degrees is $\bfe$. 

There are some natural restrictions on $L(\bfe)$ coming from the multigraded
Herzog-K\"uhl equations, drawn to 
my attention by M.Boij. If the resolution resolves the module $M$, 
the multigraded Hilbert series of $M$ is
\[ h_M(t) = \frac{\sum_{i,\bfa} (-1)^i \beta_{i, \bfa} \cdot t^ {\bfa}} 
{\Pi_{k=1}^ n (1-t_i)}. \] 
If $M$ is artinian, $h_M(t)$ is a polynomial and 
\begin{equation} \label{SetLigBeta} 
\sum_{i,\bfa} (-1)^ i \beta_{i, \bfa} t^ {\bfa} = h_M(t) \cdot \Pi_{k=1}^n(1-t_i). 
\end{equation}
For each multigraded $\bfa \in \hele^ {n}$ and integer $k = 1, \ldots, n$, let 
the projection $\pi_k(\bfa)$ be $(a_1,\ldots, \hat{a}_k, \ldots, a_n)$, 
the $(n-1)$-tuple where we omit $a_k$.

We obtain the multigraded analogs of the Herzog-K\"uhl (HK) equations
by setting $t_k = 1$ in (\ref{SetLigBeta}) for each $k$. 
This gives for every $\hat{\bfa}$ in $\hele^{n-1}$ and $k = 1, \ldots, n$ 
an equation
\begin{equation} \label{SetLigHK}
\sum_{i,\pi_k(\bfa)= \hat{\bfa}} (-1)^i \beta_{i, \bfa} = 0.
\end{equation}

Now let $L^\prime(\bfe)$ be the linear space of elements in $\oplus_{\bfa \in \hele^{n}} 
{\mathbb Q}^{n+1}$ which fulfil the multigraded HK-equations above, and
become pure diagrams when taking total degrees, and with the difference sequence of 
these total degrees equal to $\bfe$. There is a natural injection 
$L(\bfe) \pil L^\prime(\bfe)$. Note that $L^\prime(\bfe)$ does not
depend on the field $\kr$, but $L(\bfe)$ does.  
Our second main Theorem 
\ref{LinbettiTheMainTo} states that this map is an isomorphism
in characteristic $0$.

\subsection{The equivariant resolution and Schur polynomials}
In \cite{EFW} the author together with D.Eisenbud and J.Weyman
constructed a $\GL(n)$-equivariant pure resolution of an artinian module, 
whose form we now describe. For a partition $\lambda = (\la_1, \ldots,
\la_n)$ let $S_\lambda$ be the associated Schur module, which is an
irreducible representation of $\GL(n)$ (see for instance \cite{FuH}).
The action of the diagonal matrices in $\GL(n)$ gives a decomposition of 
$S_\lambda$ as a $\hele^n$-graded vector space. The basis elements are
given by semi-standard Young tableau of shape $\la$ with entries from 
$1,2, \ldots, n$. All the nonzero graded pieces in this decomposition
 have total degree $|\la| = \sum_{i=1}^n \la_i$. 
The free module $S \te_k S_\la$ then becomes a free multigraded module
where the generators all have total degree $|\la|$.

Now given the difference vector $\bfe$, let 
\[ \la_i = \sum_{j = i+1}^n (e_j - 1) \] and define a sequence of 
partitions for $i=0, \ldots, n$ by
\begin{equation} \label{SetLigAlfa}
\alpha(\bfe,i) = (\la_1 + e_1, \la_2 + e_2, \ldots, \la_i + e_i, \la_{i+1}, 
\ldots, \la_n).
\end{equation}
The construction in  \cite{EFW} then gives a $\GL(n)$-equivariant resolution
\begin{equation} \label{SetLigEe}
E(\bfe) :  S \te_k S_{\alpha(\bfe,0)} \vpil
S \te_k S_{\alpha(\bfe,1)} \vpil \cdots \vpil S \te_k S_{\alpha(\bfe,n)}
\end{equation}
of an artinian $S$-module. Note that our notation differs somewhat from
\cite{EFW}. There the $\alpha$'s depend on $\bfd$ while we use the
difference vector as argument.

The Betti polynomial of  $S \te_k S_\la$ will be the character 
of $S_\la$ which
is the Schur polynomial $s_\la$. For a matrix $(a_{ij})$ where 
$i,j = 1,\ldots, n$, let $|a_{ij}|$ denote the determinant of the matrix.
The Schur polynomial is then given by the expression 
\[ s_\la = \frac{|t_i^{\la_j + n-j}|}{|t_i^{n-j}|}. \]
Note that the denominator here is $D = \Pi_{i < j} (t_j - t_i)$.

\medskip
It is also interesting to note the following.
\begin{lemma} \label{linbettiLemMat}
For $i = 0, \ldots, n$, the $i$'th Betti polynomial $B_i$ associated to
the equivariant complex $E(\bfe)$, is the maximal minor obtained by
deleting column $n-i$ in the $n \times (n+1)$ matrix 
\begin{equation*} \label{SetLigMat}
\begin{bmatrix} 1 & t_1^{e_n} & t_1^{e_n + e_{n-1}} & \cdots 
                  & t_1^{e_n + e_{n-1} + \cdots + e_1} \\
                1 & t_2^{e_n} & t_2^{e_n + e_{n-1}} & \cdots 
                  & t_2^{e_n + e_{n-1} + \cdots + e_1} \\

                \vdots & & & &  \\
                1 & t_n^{e_n} & t_1^{e_n + e_{n-1}} & \cdots 
                  & t_n^{e_n + e_{n-1} + \cdots + e_1} 
\end{bmatrix}
\end{equation*}
divided by $D = \Pi_{i < j} (t_j - t_i)$.
\end{lemma}

\begin{proof}
Let $\rho = (n-1, n-2, \cdots, 1,0)$. 
The partition $\alpha(\bfe,i)$ is then
\[ (\sum_1^n e_j, \sum_2^n e_j, \ldots, \sum_i^n e_j, \sum_{i+2}^n e_j,
\ldots, 0) - \rho\]
and the associated Schur polynomial is then the minor we get in the matrix
above by omitting column $n-i$, and dividing by $D$.
\end{proof}

\subsection{The linear space $L(\bfe)$}
\label{LinbettiSusecLe}
For a multigraded Betti diagram 
$\beta = \{ \beta_{i, \bfa} \}$ and a multidegree $\bfl$ in $\hele^ {n}$, 
we get the twisted Betti diagram $\beta(-\bfl)$ which in homological degree
$i$ and multidegree $\bfa$ is given by $\beta_{i,\bfa-\bfl}$. If $\cF$
is a resolution with Betti diagram $\beta$, then $\cF(-\bfl)$ is
a resolution with Betti diagram $\beta(-\bfl)$. 

Also let $F_r : S \pil S$ be the map sending $x_i \mapsto x_i^r$. 
Denote by $S^{(r)}$ the ring $S$ with the $S$-module structure
given by $F_r$.  Given any complex $\cF$ we may tensor it with 
$- \te_S S^{(r)}$ and get a complex we denote by $\cF^{(r)}$. 
Note that if $\cF$ is pure with degrees $\bfd$, then $\cF^{(r)}$ is
pure with degrees $r \cdot \bfd$. 

\medskip
The following are our main results and shows that the numerical 
part of the equivariant complex, its multigraded Betti diagram, 
plays the fundamental role when considering multigraded Betti diagrams
of resolutions of artinian $\hele^n$-graded modules.

\begin{theorem} \label{LinbettiTheMainEn}
Let $r = \gcd(e_1, \ldots, e_n)$ and let $\bfe = r \cdot \bfe^\prime$. 
The Betti diagrams $\beta_{E(\bfe^\prime)^{(r)}}(\bfa)$ where $\bfa$ varies
over $\hele^n$, form a basis for the lattice of integral points in
$L^\prime(\bfe)$. 
Moreover $\beta_{E(\bfe^\prime)^{(r)}}$ is,
up to sign and twist with $\bfa \in {\hele}^n$, the unique element
in $L^\prime(\bfe)$ with this property.

   In particular, the $\beta_{E(\bfe^\prime)^{(r)}}(\bfa)$ where $\bfa$ varies
over $\hele^n$ form basis for the vector space $L^\prime(\bfe)$.
\end{theorem}

The proof will be given in Section \ref{LinHKSec}. 
The first part of Theorem \ref{LinbettiTheMainEn}
may also be formulated in an equivalent way in terms of the associated
$(n+1)$-tuple of Betti polynomials introduced at the end of 
Subsection \ref{SetSubsHK}.

\medskip
\noindent {\bf Theorem \ref{LinbettiTheMainEn}${}^\prime$} {\em
Let $s = (s_0, \ldots, s_n)$ be the $(n+1)$-tuple of Betti
polynomials of $E(\bfe^\prime)^{(r)}$. If $B = (B_0, \ldots, B_n)$ is any
$(n+1)$-tuple of homogeneous Laurent polynomials 
fulfilling the HK-equations (\ref{SetLigHK}), 
and where the difference vector of the total
degrees is $\bfe$, then $B = p \cdot s$ for some homogeneous
Laurent polynomial $p$.}

\medskip
\begin{example}
Letting $B_{(1)}$ and $B_{(2)}$ be the triples of Betti polynomials of the 
resolutions (\ref{IntroLigEkvi}) and (\ref{IntroLigBS}) of the example in 
the introduction, we have
\[ B_{(2)} = (t_1^2 - t_1t_2 + t_2^2) B_{(1)}. \]
\end{example}

Letting ${\QQ}\{t_1, \ldots, t_n\}$ be the Laurent polynomial ring,
we see that $L^\prime(\bfe)$ is the free module 
${\QQ}\{t_1, \ldots, t_n\}\cdot s$. Identifying  $L^\prime(\bfe)$ 
with this Laurent polynomial ring we have:

\begin{theorem} \label{LinbettiTheMainTo}
The image of the map $L(\bfe) \pil L^\prime(\bfe)$ is an
ideal in the Laurent polynomial ring. When  $\kr$ has characteristic $0$ or
$n=2$, this map is an isomorphism.
\end{theorem}

The ideal, which is the image of the above map, depends on $\bfe$ and 
may depend on the field $\kr$; denote it ${\mathscr I}(\kr;\bfe)$.
In the case when $\kr$ has characteristic $p$ and $n \geq 3$, it is an 
interesting question to determine this ideal.

\begin{question} Is the ideal ${\mathscr I}(\kr;\bfe)$ in 
the Laurent polynomial ring ${\QQ}\{t_1, \ldots, t_n\}$ always
nonzero? Is it always equal to the whole ring? 
\end{question}

\subsection{The associated diagrams when taking total degrees}

On the rational rays of pure $\hele$-graded Betti diagrams, it is an
open question what integral points come from actual pure resolutions.
The following says that in the case of diagrams arising from 
$\hele^n$-graded resolutions of {\it artinian} modules over 
$\kr[x_1, \ldots, x_n]$, we will not get more than what we get from 
the equivariant resolution. See however the following remark.

\begin{corollary} \label{linbettiCorSimgr}
 (Char.$\,\,\kr = 0$.) Let $\pi$ be the $\hele^n$-graded Betti diagram of a
$\hele^n$-graded artinian module over $\kk[x_1, \ldots, x_n]$, 
whose resolution becomes pure when taking
total degrees. Then the associated $\hele$-graded Betti diagram
$\overline{\pi}$ is an integer multiple of the associated
$\hele$-graded Betti diagram of the equivariant resolution, suitably twisted.
\end{corollary}

\begin{proof} By Theorems \ref{LinbettiTheMainEn} and \ref{LinbettiTheMainTo},
$\pi$ is a linear combination 
$\sum_i k_i \beta_E(\bfa_i)$ where all the $\bfa_i$ have the 
same total degree, say $a$. Let $P$ be the $(n+1)$-tuple of Laurent
polynomials associated to $\pi$, and $s$ the associated $(n+1)$-tuple
of Laurent polynomials associated to the equivariant resolution.
Then 
\[ P = \sum_i k_i t^{{\bfa}_i}s, \]
and we will show that all coefficients here are integers. 
Considering  the first polynomial in the tuple we have
$P_0 = \sum_i k_i t^{{\bfa}_i}s_0$. 
Since the highest weight vectors of Schur modules have multiplicity one,
the lexicographically largest term of $s_0$ has 
coefficient $1$. 
The coefficient of the highest lexicographic term
of $P_0$ must then equal $k_i$ for some $i$, and so $k_i$ is an integer.
Then $P_0 - k_i t^{{\bfa}_i}s_0$ has integer coefficients. In this way
we may continue and get that  all $k_j$ are integers.
Taking total degress we get
\[ \overline{\pi} = (\sum_i  k_i)\, \,  \overline{\beta_E}(a) . \]
\end{proof}

\begin{remark} On the ray generated by the diagram
\[ \tau =  \begin{pmatrix} 1 & 2 & - & - \\ - & - & 2 & 1 
   \end{pmatrix} , \]
 the equivariant diagram is $3 \tau$. The above says that it is not
possible to realize $2 \tau$ (or $5\tau$ or $7 \tau$) 
as coming from a $\hele^3$-graded diagram over 
the polynomial ring in three variables. It is however possible to realize
$2 \tau$ as coming from a $\hele^4$-graded diagram over the polynomial ring 
$S$ in four variables. Just take a general $\hele^4$-graded map  :
\[ S^2_{(0,0,0,0)} \leftarrow S_{(1,0,0,0)} \oplus S_{(0,1,0,0)} \oplus
S_{(0,0,1,0)} \oplus S_{(0,0,0,1)}. \] Then $2 \tau$ will be
the $\hele$-graded diagram of the resolution  of the cokernel. Note
however that the cokernel is not artinian.
\end{remark}

\section{Schur polynomials}

We describe the greatest common divisor of 
the Betti polynomials occurring in the equivariant pure resolutions.
We do this in Theorem \ref{SchurTheFelles} and this is the only
result of this section that we use later on.

\subsection{Common divisors and group actions}
Suppose a group $G$ acts on the polynomial ring $\kk[t_1, \ldots, t_n]$.
A polynomial $p$ is {\it semi-invariant} if the groups acts
as $g.p = \mu(g)p$ for some character $\mu : G \pil \kk$. 

\begin{lemma}
Let $p$ and $q$ be semi-invariant polynomials in $\kk[t_1, \ldots, t_n]$.
Then their greatest common divisor is also a semi-invariant.
\end{lemma}

\begin{proof}
If $b$ is the greatest common divisor, then $g.b$ is also a common divisor.
Hence $g.b = \mu(g)b$ for some character $\mu$.
\end{proof}

We now give two cases where one may actually conclude that if $p$ and $q$ are
invariants, their greatest common divisor is also an invariant.
Recall the algebra morphism $F_r : S \pil S$ of Subsection 
\ref{LinbettiSusecLe}.

\begin{lemma} \label{SchurLemFrFelles}
Let $p$ and $q$ be polynomials in $\kk[t_1, \ldots, t_n]$ and 
$b$ their greatest common divisor. For a natural number $r$, 
the greatest common divisor
of $p^{(r)}$ and $q^{(r)}$ is $b^{(r)}$.
\end{lemma}

\begin{proof}
The group $(\hele_r)^n$ acts on the polynomial ring, and
$p^{(r)}$ and $q^{(r)}$ are invariants. 
Note that any semi-invariant polynomial for this group has the form
$m \cdot c^{(r)}$ for some monomial $m = x_1^{a_1} \cdots x_n^{a_n}$
where each $0 \leq a_i < r$ and this monomial is uniquely determined by the 
character. Write $p = m_1 p_1$ where $m_1$ is a monomial and $p_1$ does not
have any monomial as a factor, and similarly $q = n_1 q_1$. Let 
$m c^{(r)}$ be the greatest common divisor of $p^{(r)}$ and $q^{(r)}$ 
where $m$ is a monomial and $c^{(r)}$ does not have a monomial factor. Then 
$m$ divides $m_1^{(r)}$ and $n_1^{(r)}$, and $m_1^{(r)}/m$ and $n_1^{(r)}/m$
are semi-invariants with the same character. If this character is non-trivial
they will have a common monomial factor. But this is impossible by choice of $c$. Hence
$m$ is also an invariant.
\end{proof}


\begin{lemma} \label{SchurLemSyFelles}
The greatest common divisor of two symmetric polynomials in
$\kk[t_1, \ldots, t_n]$ is also a symmetric polynomial.
\end{lemma}

\begin{proof}
The symmetric group $S_n$ has two characters, the trivial one and the sign of
the permutation. If the greatest common divisor $f$ is not symmetric then 
$\sigma \cdot f = (-1)^{\sgn(\sigma)} f$. Hence $f$ is divisible by $t_i - t_j$ 
for each pair $i < j$ and so by $D = \Pi_{i < j} (t_j - t_i)$. But then both
$p/D$ and $q/D$ are semi-invariants with the sign character, and so are again 
divisible by $D$. Thus $f = D^2 f^\prime$ where $f^\prime$ is a greatest common divisor
of $p/D^2$ and $q/D^2$. By induction on degree we may assume that $f^\prime$ is 
symmetric.
\end{proof}

\subsection{Common divisors of Schur polynomials}


For a polynomial $f$ in $\kk[t_1, \ldots, t_n]$ write
\[ f = t_1^N \overline{f} + \mbox{ lower terms in }t_1  + 
t_1^n \underline{f}\]
where the last term is the one with the smallest power of $t_1$.
The polynomials $\overline{f}$ and $\underline{f}$ are in 
$\kk[t_2, \ldots, t_n]$.
Note that if $f = gh$ then $\overline{f} = \overline{g} \overline{h}$
and $\underline{f} = \underline{g} \underline{h}$.
For a partition $\la = (\la_1, \ldots, \la_n)$ let 
\begin{eqnarray}
\overline{\la} & = & (\la_2, \ldots, \la_n) \label{SchurLigOver}\\
\underline{\la} & = & (\la_1-\la_n, \ldots, \la_{n-1} - \la_n). \notag
\end{eqnarray}
By the way Schur polynomials are derived from semi-standard Young tableaux,
we see that
\[ \overline{s_\la} = s_{\overline \la}, \quad \underline{s_\la} 
= (t_2\cdots t_n)^{\la_n} s_{\underline{\la}}. \]

\begin{example} \label{schurEksS}
Let $n=3$ and $\lambda = (4,2,1)$. Then
\begin{eqnarray*}
s_{4,2,1} &=& t_1^4t_2^2t_3 + t_1^4t_3^2t_2 + t_2^4t_3^2t_1 + 
               t_2^4t_1^2t_3 + t_3^4t_1^2t_2 + t_3^4t_2^2t_1 \\
    &+& t_1^3t_2^3t_3 + t_1^3t_3^3t_2 + t_2^3t_3^3t_1 +
2t_1^3t_2^2t_3^2 + 2t_2^3t_1^2t_2^2 + 2t_3^3t_1^2t_2^2.
\end{eqnarray*}
We get 
\begin{eqnarray*} 
\overline{s_{4,2,1}} &=& t_2^2t_3 + t_3^2t_2 = s_{2,1} \\
\underline{s_{4,2,1}} &=& t_2^4t_3^2 + t_3^4t_2^2 + t_2^3t_3^3 = (t_2t_3)s_{3,1}
\end{eqnarray*}
\end{example}

We shall use the notation 
\[ \xi_a(t_1,t_2) = t_1^{a-1} + t_1^{a-2}t_2 + \cdots + t_2^{a-1} \]
which factors as $\Pi_\omega (t_1 - \omega t_2)$ where the product is 
over all $a$'th roots of $1$ except $1$ itself. Note that this is equal to the 
Schur polynomial $s_{a-1,0}$. Finally let
\[ \rho = (n-1,n-2, \ldots, 1,0), \quad \rho^\prime = (n-2, n-3, \ldots, 1,0).\]

\begin{lemma} \label{SchurLemFrfaktor} 
Let $f$ be a symmetric polynomial having a non-trivial common factor 
with $s_{r\rho - \rho}$. Then $s_{p\rho - \rho}$ will divide $f$
for some divisor $p \geq 2$ of $r$.
\end{lemma}

\begin{proof}
We have
\[ s_{r\rho - \rho} = \Pi_{i<j} (t_i^r - t_j^r)/\Pi_{i < j} (t_i - t_j)
= \Pi_{i<j} \xi_{r} (t_i,t_j). \]
Suppose, say, $t_1 - \omega t_2$ is a common factor where $\omega \neq 1$
is a primitive  $p$'th root of unity where $p \geq 2$ divides $r$.
Writing 
\[ f = \Sigma_{\bfa \in \hele^{n-2}} t_3^{a_3} \cdots t_n^{a_n}
p_{\bfa}(t_1,t_2) \]
where the $p_{\bfa}(t_1,t_2)$ are symmetric polynomials over $\hele$,
we see that $t_1 - \omega t_2$ is a factor of each $p_{\bfa}(t_1, t_2)$. 
Hence 
$\xi_{p}(t_1,t_2)$ is a factor of $f$. Since $f$
is symmetric, all $\xi_{p}(t_i,t_j)$ must divide $f$ and so
$s_{p\rho - \rho}$ will divide $f$.
\end{proof}

\begin{lemma} \label{SchurLemER}
If $r$ is a common divisor of $\la_1, \ldots, \la_n$, write
$\la = r \cdot \la^\prime$ for a partition $\la^\prime$.
Then 
\[ s_{\la - \rho} = s_{\la^\prime - \rho}^{(r)} \cdot s_{r\rho - \rho}. \]
\end{lemma}

\begin{proof}
The following short argument was brought to our attention by J.Weyman
and S.Sam.
\[ s_{\la - \rho} = \frac{|t_i^{\la_j}|}{|t_i^{n-j}|}
= \frac{|t_i^{\la_j}|}{|t_i^{r(n-j)}|} \cdot \frac{|t_i^{r(n-j)}|}
{|t_i^{n-j}|} = s_{\la^\prime - \rho}^{(r)} \cdot  s_{r\rho - \rho}. \]
\end{proof}

\begin{lemma} \label{SchurLemFrRelprim}
For any $\la$ and $r$, the polynomials $s_{\la}^{(r)}$ and 
$s_{r\rho - \rho}$ are relatively prime.
\end{lemma}

\begin{proof}
If $n=2$ then 
\[ s_{\la_1, \la_2}^{(r)} = (t_1t_2)^{\la_2 r} \cdot s_{\la_1 - \la_2,0}^{(r)}. \]
So we must show that for any $a$, the polynomials $s_{a-1,0}^{(r)}$ and
$s_{r-1,0}$ are relatively prime. These polynomials are
\[ \frac{t_1^{ar}- t_2^{ar}}{t_1^r - t_2^r} \mbox{ and }
\frac{t_1^r - t_2^r}{t_1 - t_2}. \]
Since $t_1^{ar}- t_2^{ar}$ does not have any multiple factors, these
are relatively prime.

 Let now $n \geq 3$. 
If the polynomials in the statement have a greatest
common divisor $f$, 
then $\overline{f}$ is a common divisor of 
of $\overline{s_{\la}^{(r)}} = \overline{s_{\la}}^{(r)}
= s_{\overline{\la}}^{(r)}$ and $\overline{s_{r\rho - \rho}}
= s_{r\rho^\prime - \rho^\prime}$.
By induction the greatest common divisor of these is $1$. If $f$ is not
$1$ it has by Lemma \ref{SchurLemFrfaktor} a  factor of the form  
$s_{p\rho- \rho}$ for some $p \geq 2$ dividing $r$. But then 
$s_{p \rho^\prime - \rho^\prime}$ would be a factor of $\overline{f} = 1$. 
\end{proof}

\begin{lemma} \label{SchurLemLaRRelprim} Suppose $\la - \rho$ is non-negative
and $r$ is relatively prime to at least one $\la_i - \la_{i+1}$
where $1 \leq i \leq n-1$. Then $s_{\la - \rho}$ and $s_{r\rho - \rho}$
are relatively prime.
\end{lemma}

\begin{proof}
If $n = 2$, then 
\[ s_{\la_1 - 1, \la_2} = (t_1t_2)^{\la_2} s_{\la_1 - \la_2 - 1, 0}. \]
This is relatively prime to $s_{r-1,0}$, since $\xi_{p}(t_1,t_2)$
and $\xi_{q}(t_1,t_2)$ are relatively prime when $p$ and $q$ are.

Let $n \geq 3$. If $s_{\la - \rho}$ and $s_{r \rho - \rho}$ have a non-trivial 
common 
factor they have a common factor $s_{p\rho - \rho}$ where $p \geq 2$ divides $r$. 
Suppose that $p$ is relatively prime to $\la_1 - \la_2$.  
Then $\underline{s_{p\rho- \rho}}$, which is 
$(t_2 \cdots t_n)^{p-1} \cdot s_{p\rho^\prime - \rho^\prime}$,
is a common factor of $\underline{s_{\la - \rho}}$ and $\underline{s_{r\rho- \rho}}$ 
which are respectively \[ (t_2 \cdots t_n)^{\la_n - 1} \cdot  
s_{\underline{\la} - \rho^\prime} \mbox{ and } 
(t_2 \cdots t_n)^{r-1} \cdot  s_{r\rho^\prime - \rho^\prime}. \]
Then $s_{p\rho^\prime - \rho^\prime}$ would have to be a common factor of 
$s_{\underline{\la} - \rho^\prime}$ and $s_{r\rho^\prime - \rho^\prime}$ which by 
induction is not possible.

Now assume that $p$ is relatively prime to $\la_i - \la_{i+1}$ for some 
$2 \leq i \leq n-1$. Then $\overline{s_{p\rho- \rho}} = s_{p \rho^\prime - 
\rho^\prime}$ is a common factor of $s_{\overline{\la} - \rho^\prime}$ and
$s_{r\rho^\prime - \rho^\prime}$. But by induction these two latter 
polynomials are relatively prime, so again we get a contradiction.
\end{proof}

Now we are ready to prove the main result of this subsection.

\begin{theorem} \label{SchurTheFelles}
Let $r$ be the greatest common divisor of $e_1, \ldots, e_n$. Then
$s_{r \rho - \rho}$ is the greatest common divisor of 
\begin{equation} \label{SchurLigSalfa}
s_{\alpha(\bfe, 0)},\, s_{\alpha(\bfe,1)}, \, \ldots, \,  s_{\alpha(\bfe, n)}.
\end{equation}
Letting $e_i = r e_i^\prime$ we have 
\[ s_{\alpha(\bfe, i)} = s_{r \rho - \rho} \cdot s_{\alpha(\bfe^\prime,i)}^{(r)}. \] 
\end{theorem}


\begin{proof}
1. The last equation is by Lemma \ref{SchurLemER}. To show the statement, it
is by Lemma \ref{SchurLemFrFelles} enough to show that if the greatest
common divisor  $r = 1$, then the greatest common divisor of 
the $s_{\alpha(\bfe,i)}$'s is $1$. We do this by induction on $n$.
When $n = 2$, the Schur polynomials are
\[ s_{(e_2 - 1, 0)},\quad  s_{(e_1 + e_2 -1,0)}, \quad s_{(e_1 + e_2 -1, e_2)}
= (t_1t_2)^{e_2} s_{(e_1-1,0)}. \]
The first polynomial is $\xi_{e_2 }$ and the last polynomial is 
$\xi_{e_1 }$, and these are relatively prime when $e_1$ and $e_2$ are.

\medskip
\noindent 2. Suppose then $n \geq 3$. 
Let $a$ be the greatest common divisor of $e_1, \ldots, e_{n-1}$ and 
$b$ the greatest common divisor of $e_2, \ldots, e_n$. Then $a$ and $b$ are
relatively prime. Let $f$ be the greatest common divisor of the 
$s_{\alpha(\bfe, i)}$ for $i = 0,\ldots, n$. 
It is symmetric by Lemma \ref{SchurLemSyFelles}.
Also $f$ does not have any variables as a factor. Otherwise it would
be divisible by $t_1t_2 \cdots t_n$  but this does not go together with
$f$ dividing $s_{\alpha(\bfe,0)}$. 

By Lemma \ref{SchurLemER} note that 
\[ s_{\alpha(\bfe,0)} = s_{b \rho- \rho} \cdot s_{\alpha(\bfe^\prime,0)}^{(b)} \]
where $\bfe^\prime = (*, e_2/b, \ldots, e_n/b)$ (by Lemma \ref{linbettiLemMat}
the last factor
above does not depend on the first coordinate of $\bfe^\prime$).
Also 
\[ s_{\alpha(\bfe,n)} = s_{a \rho-\rho} \cdot (t_1\cdots t_n)^{e_n}
\cdot s_{\alpha(\bfe^{\prime\prime}, 0)}^{(a)} \]
where $\bfe^{\prime\prime} = (*, e_1/a, \ldots, e_{n-1}/a)$ (again by
Lemma \ref{linbettiLemMat}
the last factor above does not depend on the first coordinate of 
$\bfe^{\prime\prime}$). 
Now $f$ will be relatively prime to $s_{b\rho- \rho}$ and $s_{a\rho - \rho}$. 
This is so since $s_{b\rho- \rho}$ is relatively prime to $s_{\alpha(\bfe,n)}$
by Lemma \ref{SchurLemLaRRelprim} because 
$\alpha(\bfe,n)_1 - \alpha(\bfe,n)_2 + 1 = e_1$,
and since $s_{a\rho - \rho}$ is relatively prime to $s_{\alpha(\bfe,0)}$ by Lemma
\ref{SchurLemLaRRelprim} because $\alpha(\bfe,0)_{n-1} - \alpha(\bfe,n)_n + 1 
= e_n$.
Hence we may conclude that $f$ is a common factor of 
$s_{\alpha(\bfe^\prime,0)}^{(b)}$ and $s_{\alpha(\bfe^{\prime\prime}, 0)}^{(a)}$.

\medskip
\noindent 3. Now we consider $\overline{f}$. It divides
\[ \overline{s_{\alpha(\bfe^\prime,0)}^{(b)} } = 
\overline{s_{{\alpha(\bfe^\prime,0)}}}^{(b)} = 
s^{(b)}_{\overline{{\alpha(\bfe^\prime,0)}}} =
s_{\alpha(\overline{\bfe}^\prime , 0)}^{(b)} \]
where $\overline{\bfe}^\prime = (*,e_3/b, \ldots, e_n/b)$. 
But it also divides the $\overline{s_{\alpha(\bfe,i)}}$ and for $i \geq 1$ these
are by (\ref{SetLigAlfa}) and (\ref{SchurLigOver}) equal to
\[ s_{\alpha(\overline{\bfe},i)}, \quad i \geq 1 \]
where $\overline{\bfe} = (e_2, \ldots, e_n)$. By induction the greatest common divisor
of these polynomials is $s_{b\rho^\prime - \rho^\prime}$. We may then by Lemma 
\ref{SchurLemFrRelprim} conclude that $\overline{f} = 1$.
Since $f$ is symmetric we have
\[ f = t_1^m + \mbox{ lower terms in } t_1 + \underline{f} \]
where $\underline{f}$ 
does not have any variable as a factor.

\medskip 
\noindent 4. Now consider $\underline{f}$. 
We know that $f$ divides 
$s_{\alpha(\bfe^{\prime\prime}, 0)}^{(a)}$.
Note that 
\[ \underline{s_{\alpha(\bfe^{\prime \prime}, 0)}} = 
(t_2 \cdots t_n)^{{\bfe}^{\prime\prime}_n - 1} 
\cdot s_{\alpha(\underline{\bfe}^{\prime \prime},0)} \]
where $\underline{\bfe}^{\prime \prime} = (*, e_1/a, \ldots, e_{n-2}/a)$. 
Hence $\underline{f}$ divides $s_{\alpha(\underline{\bfe}^{\prime \prime},0)}^{(a)}$.

But $\underline{f}$ also divides
$\underline{s_{\alpha(\bfe, i)}}$ for $i = 0, \ldots, n-1$ which is 
\[ (t_2 \cdots t_n)^{e_n - 1}\cdot s_{\alpha(\hat{\bfe},i)} \]
where $\hat{\bfe} = (e_1, \ldots, e_{n-1})$. 
Hence $\underline{f}$ is a common factor of the $s_{\alpha(\hat{\bfe}, i)}$.
By induction their greatest common divisor is $s_{a\rho^\prime - \rho^\prime}$.
By Lemma \ref{SchurLemFrRelprim} we may now conclude that $\underline{f} = 1$. 
Since $f$ is symmetric we may further conclude that $f = 1$.
\end{proof}

\section{The linear space of diagrams fulfilling the Herzog-K\"uhl  equation}

\label{LinHKSec}

The theorem below provides a nice structural description of the $\rasj$-vector
space $L^\prime(\bfe)$ of diagrams fulfilling the HK-equations (\ref{SetLigHK}).
Recall again that $L^\prime(\bfe)$ does not depend on the field $\kk$. 
Using this theorem and Theorem \ref{SchurTheFelles}, which gives the common 
factor of the Schur polynomials, the proofs of our main Theorems
\ref{LinbettiTheMainEn} and \ref{LinbettiTheMainTo} are rather immediate.

\begin{theorem} \label{LinHKTheLprim}
Let $\bfe = (e_1, \ldots, e_n)$ be a vector of positive integers.

a. There is an  $(n+1)$-tuple $(A_0, \ldots, A_n)$ of homogeneous Laurent 
polynomials in $n$ variables, such that $L^\prime(\bfe)$ has a basis consisting 
of all $t^\bfa (A_0, \ldots, A_n)$ where $\bfa$ varies in $\hele^n$. 

b. The $A_i$'s have no common factors except for units
(which are products of nonzero constants and Laurent monomials
$t^\bfa$), and are uniquely determined up to common multiplication by
a unit.
\end{theorem}

We prove this towards the end of this section. As a consequence
of the above theorem we can prove our main theorems.

\begin{proof}[Proof of Theorem \ref{LinbettiTheMainEn}.]
Let $\bfe = r \cdot \bfe^\prime$. The associated $(n+1)$-tuple of Betti
polynomials to the complex $E(\bfe^\prime)^{(r)}$ is ( when $\kr$ has
characteristic $0$) 
\[ s = (s_{\alpha(\bfe^\prime,0)}^{(r)}, \ldots, s_{\alpha(\bfe^\prime,n)}^{(r)}).
\]
This will be a multiple of $(A_0, \ldots, A_n)$ in  Theorem \ref{LinHKTheLprim} 
above. By Theorem \ref{SchurTheFelles} the greatest common divisor of these
Schur polynomials is $1$. Hence we can take them to be equal to the $A_i$'s.
So if $B$ is in $L^\prime(\bfe)$, then $B = ps$ for some
Laurent polynomial $p$. And if $B$ is integral, it follows by the
same argument as in Corollary \ref{linbettiCorSimgr},
that $p$ must have integer coefficients,
proving the first part of Theorem \ref{LinbettiTheMainEn}.

To prove the second part of Theorem \ref{LinbettiTheMainEn}, 
note that if $A^\prime$ is another element 
in $L^\prime(\bfe)$ 
with the property of $A$, it must be $\gamma t^{\bfa} A$ 
for some rational number $\gamma$ and $\bfa$ in $\hele^n$. But when the
coefficients of the polynomials in $A^\prime$ are
integers then  $\gamma$ must be an integer, since
the highest weight of the Schur modules $S_{\alpha(\bfe^\prime,i)}$
always occurs with multiplicity one. And if $\gamma t^{\bfa} A$ 
is part of a lattice basis, only the values $\gamma = \pm 1$ can
occur.
\end{proof}

\begin{proof}[Proof of Theorem \ref{LinbettiTheMainTo}.]
If $E$ is a resolution with Betti polynomials $B$, then $E(\bfa)$ has
Betti polynomials $t^{\bfa} B$. Hence the image of $
L(\bfe) \pil L^\prime(\bfe)$, where the latter is 
${\QQ}\{t_1, \ldots, t_n \} \cdot s$,
identifies as an ideal in the Laurent polynomial ring. 

If $\kr$ has characteristic  $0$, the generator $s$ is in the image so the 
map is an
isomorphism. If $n=2$ it is shown in \cite{BF}, Proposition 3.1, that there
exists resolutions with Betti polynomials $s$, regardless of $\kr$. 
\end{proof}

\subsection{Properties  of tuples fulfilling the HK-equations}
\label{LinbettiSubsecHK}
\begin{proposition} Let $B = (B_0, \ldots, B_n)$ be 
an homogeneous $(n+1)$-tuple in $L^\prime(\bfe)$. 

a. If $B_0 = 0$ then $B = {\mathbf 0}$.

\noindent Assume $B_0$ is nonzero. Let $B_0 = t_1^{b_1} B_0^* + $ 
lower terms in $t_1$,
where $B_0^\prime$ is a Laurent polynomial in $t_2, \ldots, t_n$.

b. Then $B_1 = t_1^{b_1 + e_1} B_0^* +  $
lower terms in $t_1$. 

c. Each $B_i$ is nonzero and the highest power of $t_1$ occurring in 
$B_i$ for $i \geq 1$ is $t_1^{b_1 + e_1}$.
\end{proposition}

\begin{proof} Note that the statement holds when $n=1$. We shall then
work further using induction.
Let $u$ be the smallest index such that $B_u  \neq 0$. Suppose
$B_u = t_1^{p_1} B_u^* + $ lower terms in $t_1$. 
Let $t_2^{p_2}\cdots t_n^{p_n}$ be a monomial in $B_u^*$. 
Since the total degree of $B_u$ is fixed equal to 
$d_u$, then in $B_u$ this term occurs only with $t_1^{p_1}$ as the power of 
$t_1$. So let $B_u = c_{\bfp} t_1^{p_1} \cdots t_n^{p_n} + $ other terms, where
the coefficient $c_{\bfp}$ is nonzero.

The Herzog-K\"uhl equations give, by the projection omitting the first
coordinate, that some $B_{v}$, where $v > u$, 
will contain the monomial $t_1^{p_1 + d_v - d_u} 
t_2^{p_2} \cdots t_n^{p_n}$ (denoting the degree of $B_i$ by $d_i$). 

Let $A_1$ be the highest power of $t_1$ occurring in any $B_i$. 
Write $B = t_1^{A_1}B^\prime +$ lower terms in $t_1$, where
$B^\prime$ is a nonzero homogeneous $(n+1)$-tuple in the variables 
$t_2, \ldots, t_n$. 
By what we have shown $A_1 \geq p_1 + d_v - d_u$. This gives
$A_1 > p_1$ and 
the smallest $u^\prime$ for which $B^\prime_{u^\prime}$ is nonzero must be 
$> u$.  By omitting $B_0^\prime$
(which is zero)
we may consider $B^\prime$ as an $n$-tuple
in $t_2, \ldots, t_n$. Also we see that it will satisfy the 
Herzog-K\"uhl equations for $n$-tuples, by looking at the equations
satisfied by $B$ when we always keep the first coordinate equal
to ${A_1}$. By induction on $n$ we get that $B_1^\prime$ is nonzero and 
so the index $u$ must be $0$. This proves a. and shows that $b_1 = p_1$.
Also, by induction from c. we get that each $B_i^\prime$ is nonzero,
for $i \geq 1$. Hence we get $B_i$ nonzero for $i \geq 1$, and we also have shown $B_0$  
nonzero, proving the first part of c.

Let $t_2^{q_2} \cdots t_n^{q_n}$ be a term occurring in $B_1^\prime$.
The Herzog-K\"uhl equations with projection omitting the first coordinate,
gives that $t_1^{A_1 - e_1}t_2^{q_2} \cdots t_n^{q_n}$ occurs as a term in 
$B_0$. Hence $A_1 - e_1 \leq b_1$. Since we also have $A_1 \geq b_1 + e_1$
we get $A_1 = b_1 + e_1$. Since the $B_i^\prime$ are nonzero this also
proves the second part of c.

Now if $c_{\bfp} t_1^{p_1} \cdots t_n^{p_n}$ is a term in $B_0$ with $p_1 = b_1$,
it follows by the HK-equations, by the projection omitting the first
coordinate, that $t_1^{p_1+\sum_{i = 1}^v e_i}t_2^{p_2} \cdots t_n^{p_n}$ 
occurs as a term in $B_v$ for some $v$. Since $A_1 = b_1 + e_1$ this
only happens for $v = 1$ and with coefficient $c_{\bfp}$, thus proving b. 
\end{proof}

\begin{lemma} Let $p$ be a homogeneous Laurent polynomial and $B$ a homogeneous
$(n+1)$-tuple. Then $p \cdot B$ 
fulfils the HK-equations if and only if $B$ fulfils these equations.
\end{lemma}

\begin{proof}
The if direction is clear. So suppose $p \cdot B$ fulfils the HK-equations
but $B$ does not. So for some $n-1$-tuple $\hat\bfa$ the equation
(\ref{SetLigHK}) is not fulfilled. By re-indexing we may assume that
$\hat \bfa$ are the first coordinates in  $\bfa$, i.e. $k = n$
in (\ref{SetLigHK}).

Also suppose $\hat\bfa$ is the lexicographic largest $(n-1)$-tuple such that 
the HK-equations for $B$ do not hold. 
Write $p = t_1^{\lam_1} t_2^{\lam_2} \cdots t_n^{\lam_n} + $ lower terms for the
lex order. Then we see that for the $(n-1)$-tuple 
\[ (\lam_1, \lam_2, \ldots, \lam_{n-1}) + 
(\hat a_1, \hat a_2, \ldots, \hat a_{n-1}), \]
the HK-equations for $p \cdot B$ where we keep the first $n-1$ coordinates
fixed equal to the $(n-1)$-tuple above, does not hold.
\end{proof}

\begin{corollary} If $p$ is any Laurent polynomial, 
and $B$ is any $(n+1)$-tuple of Laurent polynomials, 
then $B$ is in $L^\prime(\bfe)$ if and only if 
$p \cdot B$ is in $L^\prime(\bfe)$. 
In particular $L^\prime(\bfe)$ is a submodule of $L^{n+1}$. 
\end{corollary}

\begin{proof}
This is because $L^\prime(\bfe)$ is a graded vector space.
\end{proof}

Given a Laurent polynomial $B_0$,
let $t_1^{c_1} \cdots t_n^{c_n}$ be
the lexicographic largest term in $B_0$. 
For each $i = 1, \ldots, n$ let $b_i$ be the smallest integer such that 
\[ t_1^{c_1} \cdots t_{i-1}^{c_{i-1}}t_i^{b_i} t_{i+1}^{d_{i+1}} \cdots
t_n^{d_n} \]
is in $B_0$ for some choice of $d_{i+1}, \ldots, d_n$. 
We define the {\it valuation} of $B_0$ to be 
\[ (c_1 - b_1, c_2 - b_2, \ldots, c_n - b_n). \]
\begin{example} \label{linHKEksVal}
The valuation of the Schur polynomial $s_{4,2,1}$ of 
Example \ref{schurEksS} is $(3,1,0)$.
\end{example}

The valuation is in $\nat^n$ (note however that $c_n - b_n$ is always zero). We 
now order $\nat^n$ lexicographically with $0 < 1 < 2 < \cdots$. We
may note  that $\nat^n$ with this ordering is a well-ordered set,
i.e. each subset has a smallest element.
When $B$ is an $(n+1)$-tuple of Laurent polynomials with $B_0$ nonzero,
we define the valuation of $B$ to be the valuation of $B_0$.

\subsection{The abstract situation}
To give a more transparent argument we will now abstract 
our situation.
Let $L$ be an integral domain, and $M$ a submodule of $L^{n+1}$.
Suppose we have a map $v : M \backslash \{0\} \pil T$, which we call
a valuation, to a well ordered set $T$, subject to the following
requirements.

\begin{itemize}
\item[1.] If $p$ is in $L$ and $b$ in $M$, then 
  $v(pb) \geq v(b)$.
\item[2.] If $b$ is in $L^{n+1}$ and $p$ in $L$, then $pb$ is in 
$M$ if and only if $b$ is in $M$.
\item[3.] Let $a$ and $b$ be in $M$. Then there exists nonzero
$p$ and $q$ in $L$
such that $pa - qb$ is either zero or has valuation $< \max \{ v(a), v(b) \}$.
\end{itemize}

Note that if $L$ is the Laurent polynomial ring, $M$ is $L^\prime(\bfe)$, 
and $T = \nat^n$,
by letting $v(B)$ be the valuation as defined in the end of the
preceding subsection, it fulfils 1. and 2. Note that $v$ is welldefined
since $B_0$ is nonzero if $B$ is nonzero.

   We shall later show that it fulfils 3. But let us assume that we have
a valuation as above. We then get a stronger version of 3.

\begin{lemma} \label{LinLemMinMax}
Given a valuation $v$ as above, and let $a$ and $b$ be nonzero in $M$.
Then there are nonzero $p$ and $q$ in $L$ such that $pa - qb$ is either zero
or has valuation $ < \min \{ v(a), v(b) \}$.
\end{lemma}

\begin{proof}
Suppose $v(a) \leq v(b)$. We can then find nonzero 
$p_1$ and $q_1$ such that $b^\prime = q_1b - p_1 a$ is zero or has
valuation $< v(b)$. If $b^\prime$ is nonzero with valuation $\geq a$,
we may continue and find nonzero $p_2$ and $q_2$ such that 
\[ b^{\prime\prime} = q_2b^\prime - p_2 a = q_2q_1b - (p_1 q_2 + p_2) a  
\] is either zero or has valuation $< v(b^\prime)$. 
In this way we may continue. If the process does not stop we have an 
infinite strictly decreasing chain of valuations, 
contrary to $T$ being well-ordered.
Hence for some $n$ we obtain 
\begin{eqnarray*} 
b^{(n)} & = & q_n \cdot b^{(n-1)} - p_n \cdot a \\
  & = & q_n q_{n-1} \cdots q_1 \cdot b - p^\prime_n a
\end{eqnarray*}
(for some $p^\prime_n$)
which has valuation $< v(a)$, or is zero.
Note that $q_n q_{n-1} \cdots q_1$ is nonzero. By 1. we must also have
$p_n^\prime$ nonzero. 
\end{proof}

We are now ready to prove our structure result for valuations fulfilling
requirements 1., 2. and 3. 

\begin{proposition} \label{LinbettiProStruktur}
Suppose $L$ is a unique factorisation domain, and $M$ is a submodule of $L^{n+1}$
with a valuation fulfilling 1., 2., and 3.
Then there is an $a$ in $M$ such that $M$ is the submodule
of $L^{n+1}$ generated by $a$. Any $a$ in $M$ such that the greatest
common divisor of its components $a_1, \ldots, a_{n+1}$ is $1$, is such
a generator.
\end{proposition}

\begin{proof}
Given the first statement, the second is clear. 
Let $a$ be a nonzero element of $M$ with the smallest possible valuation.
Then we may write $a = p a^\prime$ for some polynomial $p$ where
the components of $a^\prime$ has $1$ as their greatest common divisor.
By axiom 1. 
$v(a^\prime) \leq v(a)$ and so we may assume that $a$ is $a^\prime$.
Now choose any nonzero $b$ in $M$. Then there are nonzero $p$ and $q$
such that $qb - pa$ is either zero or has valuation $< v(a)$. The 
latter is not so by assumption, so $qb = pa$. We factor out any common factors of 
$p$ and $q$. But then by unique factorisation and construction of $a$
we must have $q$ a unit. 
\end{proof}

We will now show that the property 3. holds in our case when 
$L$ is the Laurent polynomial ring in $n$ variables and $M$ is
the submodule $L^{\prime}(\bfe)$.

\begin{proposition} \label{LinbettiProLauval}
Let $A$ and $B$ be $(n+1)$-tuples in $L^\prime(\bfe)$ and $v$ the valuation
defined at the end of Subsection \ref{LinbettiSubsecHK}. Then
there are nonzero Laurent polynomials $p$ and $q$ such that 
$pA - qB$ is zero or has valuation $< \max \{ v(A), v(B) \}$.
\end{proposition}

\begin{proof}
If $n = 1$ then $A = (\alpha t^a, \alpha t^{a+e_1})$ and 
$B = (\beta t^{b}, \beta t^{b+e_1})$ so this clearly holds.

Suppose $n \geq 2$.  By adjusting $A$ and $B$ by units, actually 
Laurent monomials
$t^{\bfc}$, we may assume that the leading terms of the first polynomials
in $A$ and $B$ for the 
lex order are their valuations. (In Example \ref{linHKEksVal} this
amounts to replacing $s_{4,2,1}$ by $t_1^{-1} t_2^{-1}t_3^{-1}s_{4,2,1}.$)
Let 
\begin{eqnarray*}
A & = & t_1^{a_1} A^\prime + \mbox { lower terms in } t_1 \\
B & = & t_1^{b_1} B^\prime + \mbox { lower terms in } t_1,
\end{eqnarray*}
and assume $b_1 \geq a_1$.
Then the valuation of $A^\prime$ is the projection $\pi_1(v(A))$ and
the valuation of $B^\prime$ is $\pi_1(v(B))$. By induction on $n$ and
Lemma \ref{LinLemMinMax} we may find nonzero $p$ and $q$ in variables
$t_2, \ldots, t_n$ such that $q B^\prime - p A^\prime$ is zero or has
valuation less than that of both $\pi_1(v(A))$ and $\pi_1(v(B))$.
But then $q B - t_1^{b_1 - a_1} p A$ will have valuation less than
the maximum of $v(A)$ and $v(B)$.
\end{proof}

We may now finish off.
\begin{proof}[Proof of Theorem \ref{LinHKTheLprim}.]
Parts a. and b. follow from Proposition \ref{LinbettiProStruktur}
by letting $L$ be the Laurent polynomial ring in the variables
$t_1, \ldots, t_n$, and $v$ the valuation defined at the end of Subsection
\ref{LinbettiSubsecHK}. This is a valuation by Proposition 
\ref{LinbettiProLauval}.
\end{proof}

\section*{Acknowledgements}
I thank M.Boij and J.Weyman for discussions concerning this paper.

\bibliographystyle{mrl}
\bibliography{Bibliography}

\end{document}